\documentclass[11pt]{article}
\usepackage{amssymb,amsmath,amsthm,amsfonts, commath, mathrsfs, mathtools, secdot}
\usepackage[T1]{fontenc}		
\usepackage[a4paper, margin=1in]{geometry}

\usepackage{bbm}
\usepackage{extarrows}
\usepackage{enumerate}
\usepackage{multicol}	
\usepackage{fge}
\usepackage{oldgerm}
\usepackage{hyperref} 
\usepackage{paralist} 
\usepackage{accents}

\usepackage{tikz}
\usetikzlibrary{datavisualization}
\usetikzlibrary{datavisualization.formats.functions}
\usepackage{pgfplots}
\usepackage{tkz-fct}
\pgfplotsset{every axis/.append style={
    axis x line=middle,    
    axis y line=middle,    
    axis line style={<->}, 
    xlabel={$x$},          
    ylabel={$y$},          
    },
    cmhplot/.style={color=blue,mark=none,line width=1pt,<->},
    soldot/.style={color=blue,only marks,mark=*},
    holdot/.style={color=blue,fill=white,only marks,mark=*},
}
    
\tikzset{>=stealth}

\usepackage[natbib=true, backend=biber, style=numeric]{biblatex}
\usepackage{csquotes}
\sectiondot{subsection}
\usepackage{xcolor}

\usepackage{color}

\newtheoremstyle{Assump}%
  {3pt}
  {3pt}
  {\itshape}
  {}
  {\bfseries}
  {.}
  {.5em}
  {\thmname{#1} \thmnumber{#2} \thmnote{\normalfont#3}}

\newtheorem{theorem}{Theorem}[section]

\newtheorem{prop}[theorem]{Proposition}

\theoremstyle{definition}
\theoremstyle{definition}
\theoremstyle{definition}\newtheorem{remark}[theorem]{Remark}
\theoremstyle{definition}

\numberwithin{equation}{section}

\newcommand{\R}{\mathbb{R}}

\newcommand{\D}{{\bf D}}											

\newcommand{\F}{\mathbb{F}}											
\newcommand*\diff{\mathop{}\!\mathrm{d}}								
\newcommand{\E}{\mathbb{E}}											
\newcommand{\Pro}{\mathbb{P}}										
\newcommand{\indep}{\perp \!\!\! \perp}									
\newcommand{\ind}{\operatorname{\mathbf{1}}}							
\newcommand{\convw}[1]{\xRightarrow[ {#1}]{}}							
\renewcommand{\max}[1]{\underset{#1}{\operatorname{max}}\;}				
\renewcommand{\setminus}{\mathbin{\fgebackslash}}						
    
\newcommand{\dJ}{\diff_{\operatorname{J1}}}								
\newcommand{\dM}{\diff_{\operatorname{M1}}}							

\title{Functional CLTs for subordinated L\'evy models \\in physics, finance, and econometrics}

\author{Andreas S\o jmark\thanks{London School of Economics, Department of Statistics, London, WC2A 2AE, UK. \texttt{a.sojmark@lse.ac.uk}} \,\; and \, Fabrice Wunderlich\thanks{University of Oxford, Mathematical Institute, Oxford, OX2 6GG, UK. \texttt{wunderlich@maths.ox.ac.uk}}  }

\begin{document}
\maketitle

\begin{abstract}
We present a simple unifying treatment of a broad class of applications from statistical mechanics, econometrics, mathematical finance, and insurance mathematics, where (possibly subordinated) L\'evy noise arises as a scaling limit of some form of continuous-time random walk (CTRW). For each application, it is natural to rely on weak convergence results for stochastic integrals on Skorokhod space in Skorokhod's J1 or M1 topologies. As compared to earlier and entirely separate works, we are able to give a more streamlined account while also allowing for greater generality and providing important new insights. For each application, we first elucidate how the fundamental conclusions for J1 convergent CTRWs emerge as special cases of the same general principles, and we then illustrate how the specific settings give rise to different results for strictly M1 convergent CTRWs.
\end{abstract}

\section{Introduction}

This paper develops functional CLTs for four different problems coming from statistical mechanics, econometrics, mathematical finance, and insurance mathematics. Each problem was briefly discussed in Section 2 of \cite{andreasfabrice_theorypaper}, but the detailed analysis is the topic of the present paper. Hitherto, they have been treated in completely separate strands of literature and in less generality.

Our goal here is to give a simple, concise, and unified account, noting that, for each of the applications, the analysis ultimately boils down to nuances of the same fundamental problem: namely that of weak convergence for certain stochastic integrals driven by moving averages or, more generally, continuous-time random walks (CTRWs). The starting point for this analysis is the fact that these integrators satisfy functional CLTs with limits described by stable L\'evy processes, typically time-changed by a $\beta$-stable subordinator (and possibly with a tempered jump size distribution for the parent L\'evy process). Throughout, our analysis relies only on the framework developed in \cite{andreasfabrice_theorypaper} and various generalizations thereof presented in \cite{andreasfabrice}.

In the next subsection, we briefly cover the main definitions for moving averages and CTRWs along with the functional limit theorems that are available for such processes. The rest of the paper consists of four sections dedicated to the four applications mentioned above. Throughout, we shall work on Skorokhod space, both under the J1 and the M1 topology. For a brief introduction to these topologies, the reader is invited to consult \cite[Appendix A]{andreasfabrice_theorypaper}.

\subsection{Functional CLTs for moving averages and CTRWs}\label{subsect:CTRW}

Consider i.i.d.~waiting times $J_1, J_2,...$, and let  
\begin{equation}\label{eq:CTRW_N}N(t) \; := \; \max{} \{ k \ge 0 \; : \; L(k) \le t \}
\end{equation}
where 
$L_k:= J_1+\hdots+J_k$ is the time of the $k$-th jump. We can then consider
\begin{align} X^n_t \; := \;  \frac{1}{n^{\frac \beta \alpha}} \; \sum_{k=1}^{N(nt)} \, \zeta_k ,\qquad \zeta_i \; := \; \sum_{j=0}^\infty \; c_j \theta_{i-j},\label{defi:CTRW}
\end{align}
for a sequence of i.i.d.~random variables $\theta_k$ with $-\infty < k < \infty$. We let the $\theta_k$ be in the
normal domain of attraction of a non-degenerate $\alpha$-stable random variable $\tilde{\theta}$ with $0< \alpha< 2$, and let the $J_k$ be in the normal domain of attraction of a $\beta$-stable random variable with $\beta \in (0,1)$. Then 
\eqref{defi:CTRW} is known as a \emph{correlated CTRW} (or uncorrelated if $c_j=0$ for all $j\neq 0$).

In this paper, we restrict to non-negative constants $c_j\ge 0$ for all $j\geq 0$. We shall also assume that if $1<\alpha\le 2$, then either (i) $c_j=0$ for all but finitely many $j\ge 0$, or (ii) the sequence $(c_j)_{j\ge 0}$ is monotone and $\sum_{j=0}^\infty c_j^\rho <\infty$ for some $\rho<1$. If, in \eqref{eq:CTRW_N}, we take $L_k=k$, then the process $X^n$ is known as a \emph{moving average}. For such moving averages we have $ X^n \Rightarrow (\sum_{j=0}^\infty c_j) Z$ on $(\D_{\R^d}[0,\infty), \, \dM)$ (see \cite[Thm.~4.7.1]{whitt}). This was first shown by Avram \& Taqqu \cite{avram}.

A CTRW \eqref{defi:CTRW} is called \emph{uncoupled} if the sequences $(J_i)_{i\geq 1}$ and $(\zeta_k)_{k\geq 1}$ are independent. In the uncoupled case, \cite{meerschaert2} extended the results of \cite{avram}, showing that, for $0<\alpha < 2$,
\begin{equation}\label{eq:CTRW_M1_conv}
 \; X^n \; \; \stackrel{\text{M1}}{\convw{n\to \infty}} \; \; \Bigl(\,\sum_{j=0}^\infty c_j\Bigr)Z_{D^{-1}},
\end{equation}
 where $Z$ is an $\alpha$-stable Lévy process (with $Z_1\sim \tilde \theta$) and $D^{-1}$ is the generalised inverse,
\begin{align} D^{-1}(t) \; = \; \operatorname{inf} \{ \, s\ge 0 \, : \, D(s) \, > \, t\, \}, \label{eq:generalised_inverse} \end{align} 
of a $\beta$-stable subordinator $D$. Moreover, \cite{meerschaert2} confirmed that J1 convergence fails as soon as $c_0>0$ and $c_1>0$. If $c_0>0$ is the only non-zero constant, then M1 convergence was first shown in \cite{Becker-Kern}, but \cite{henry_straka} later showed that the convergence in fact takes place in J1.

Let now $X^n$ be uncorrelated (i.e., take $c_0=1$ and $c_j=0$ for $j\geq 1$ in \eqref{defi:CTRW}). If the jump sizes and jump times form an i.i.d.~sequence $(\zeta_k,J_k)_{k\geq 1}$, but, for each $k\geq 1$, $\zeta_k$ and $J_k$ are allowed to be dependent, then the CTRW is said to be \emph{coupled}. For $Z$ and $D$ as above, \cite{henry_straka,jurlewicz} show that (uncorrelated) coupled CTRWs satisfy
\begin{equation}\label{eq:coupled_CTRW_J1_conv}
 \; X^n \; \; \stackrel{\text{J1}}{\convw{n\to \infty}} \; \; \bigl((Z^-)_{D^{-1}}\bigr)^+,
\end{equation}
where $x^-(t):=x(t-)$ and $x^+(t):=x(t+)$. Notice that, if the $X^n$ are uncoupled, then $Z$ and $D$ will be independent L\'evy processes (hence no common discontinuities a.s.), so the limit simplifies to $Z_{D^{-1}}$ \cite[Lem. 3.9]{henry_straka}, as in \eqref{eq:CTRW_M1_conv} with $c_0=1$ and $c_j=0$ for all $j\neq 0$. 

Throughout the paper, we will work on a family of filtered probability spaces $(\Omega, \mathcal{F}^n, \F^n, \Pro)$, for $n\geq 1$, where each filtration $\F^n$ is given by
\begin{align} \mathcal{F}^n_t \; := \; \sigma(\zeta^n_{N(ns)}, \, N(ns) \, : \, 0\le s \le t), \qquad t\ge 0. \label{eq:filtration_CTRWs} \end{align}
Here, $\zeta^n_{N(ns)}$ is understood to be the random variable $\sum_{k=1}^\infty \zeta^n_k \ind_{\{N(ns)=k\}}$. We stress that all of the results in subsequent sections remain valid if one expands the filtrations $\F^n$ in any way such that, for $X^n$ as in \eqref{defi:CTRW} with $c_j=0$ for $j\ge 1$, the independent increment property $X^n_t-X^n_s \indep \mathcal{F}^n_s$ is preserved for all $0\le s<t$ and $n\geq 1$.

\section{Stable L\'evy noise in statistical mechanics}\label{sect:langevin_longer}

Similarly to the work of Scalas \& Viles \cite{scalas}, which was motivated by the study of damped harmonic oscillators subject to stable L\'evy noise, we are interested in the functional weak convergence
\begin{align}\label{eq:conv_scalas} \int_0^\bullet f_n(s) \; \diff X^n_s \; \; \Rightarrow \; \; \int_0^\bullet f(s) \; \diff X_s
\end{align}
for given deterministic càdlàg functions $f_n,f:\mathbb{R}\rightarrow \mathbb{R}$ and a sequence of CTRW integrators $X^n$ converging to a subordinated stable Lévy process $X$. Taking $f_n=f$ to be a given continuous function, \cite[Thm.~4.15]{scalas} built on \cite{Becker-Kern} to establish  \eqref{eq:conv_scalas} in the M1 topology for uncorrelated, uncoupled CTRWs under the following conditions: Firstly, the i.i.d.~innovations should be symmetric around zero and in the normal domain of attraction of an $\alpha$-stable law, for $\alpha \in (1,2)$. Secondly, the i.i.d.~waiting times should be in the normal domain of attraction of a stable subordinator, for $\beta \in (0,1)$.

Based on \cite{andreasfabrice_theorypaper}, we obtain the following generalisation of the aforementioned result.

\begin{prop}\label{eq:fixed_integrand_J1}
    Let $X^n$ be an uncorrelated CTRW (i.e., as defined in \eqref{defi:CTRW} with $c_j=0$ for all $j\ge 1$) with $\alpha \in (0,2]$, $\beta \in (0,1)$ and scaling limit $X=((Z^-)_{D^{-1}})^+$ from \eqref{eq:coupled_CTRW_J1_conv} which reduces to $X=Z_{D^{-1}}$ if the CTRWs are uncoupled. Let also $f_n,f$ be such that $\dM(f_n,f) \to 0$ as $n\to \infty$ and suppose that $f$ and $X$ have no common discontinuities if $X^n$ is coupled.\\ Then, the convergence \eqref{eq:conv_scalas} holds true on $(\D_{\R}[0,\infty), \dJ)$.
\end{prop}
\begin{proof}
The convergence follows immediately from \cite[Thm. 3.6]{andreasfabrice_theorypaper} since the $X^n$ have good decompositions \cite[Thm.~3.25]{andreasfabrice_theorypaper} and the sequence does satisfy AVCI. Indeed, AVCI is obtained from \cite[Prop.~3.8]{andreasfabrice_theorypaper} due to the fact that $f$ and $X$ almost surely have no common discontinuities. While this has been assumed in the coupled case, the claim can be deduced for uncoupled CTRWs in full analogy to \eqref{eq:subordinated_levy_process_does_not_jump_at_deterministic_times}.
\end{proof}

The relevance of such extensions to larger classes of uncorrelated CTRWs (e.g.~of the coupled type and without symmetry restrictions) was discussed in the conclusion of \cite{scalas} and e.g.~in \cite[Rem.~6.4]{hahn}. Moreover, \cite{scalas} and \cite{hahn} also address the relevance of allowing for convergent sequences of adapted integrands, as opposed to deterministic functions. This is easily handled, provided the integrands satisfy the assumptions of \cite[Thm.~3.6]{andreasfabrice_theorypaper} or \cite[Prop.~3.22]{andreasfabrice_theorypaper} (see also the alternative criteria in \cite[Sect.~4.4]{andreasfabrice_theorypaper}), and that they are adapted to the filtrations \eqref{eq:filtration_CTRWs}.

While \cite{scalas} relies on the M1 topology, J1 convergence in fact holds, as per Proposition \ref{eq:fixed_integrand_J1}. If we instead work with correlated CTRWs, on the other hand, then one can only hope for M1 convergence. As shown in \cite[Prop.~4.6]{andreasfabrice_theorypaper}, correlated CTRWs no longer admit good decompositions and the associated stochastic integrals can fail to converge.  Nevertheless, one can still have convergence under suitable conditions. For example, we obtain the following result as a special case of \cite[Thm.~4.15]{andreasfabrice} for correlated CTRWs \eqref{defi:CTRW} with the sequence $\{c_k\}$ satisfying the technical condition
\begin{align} \begin{cases} \sum_{i=1}^\infty \, \sum_{j=i}^\infty c_j \, < \, \infty, \qquad \quad &\text{if } 1 < \alpha\, \le \, 2; \\ 
\sum_{i=1}^\infty  \Big( \sum_{j=i}^\infty c_j \Big)^\rho \, < \, \infty \,  \text{ for some } 0<\rho<1, \qquad \quad &\text{if } \alpha=1. \end{cases}.  \tag{TC} \label{eq:technical_condition} \end{align}
Notice that \eqref{eq:technical_condition} is e.g. satisfied if $\sum_{k=1}^\infty k \, c_k<\infty$ in the case $1<\alpha\le 2$ or $\sum_{k=1}^\infty k \, c_k^\rho <\infty$ for some $0<\rho<1$ in the case $\alpha=1$.

\begin{prop} Let the $X^n$ be uncoupled correlated CTRWs as in \eqref{defi:CTRW} with $\alpha \in (0,2]$ and $\beta \in (0,1)$. Further, suppose that \eqref{eq:technical_condition} holds, and that the law of the innovations $\mathcal{L}(\theta_0)$ satisfies $\E[\theta_0]=0$ in case $1<\alpha \le 2$ as well as it being symmetric around zero if $\alpha=1$. If $f,f_n \in \D_{\R}[0,\infty)$ are such that $\dM(f_n,f) \to 0$ as $n\to \infty$, then we have \eqref{eq:conv_scalas} with $X=(\sum_{j=0}^\infty c_j)Z_{D^{-1}}$.
\end{prop}

It can be shown that this also extends to random integrands adapted to \eqref{eq:filtration_CTRWs}, where however, broadly speaking, we must ensure a certain regularity of the integrands to be independent at any time of the future trajectory of the integrators, as detailed in \cite[Thm.~4.15]{andreasfabrice}.

The next result gives another example of what can be said for correlated CTRWs when we allow for more generality concerning the law of the innovations in return for additional regularity of the integrands. The result follows directly from \cite[Thm. 4.14]{andreasfabrice}.

\begin{prop} \label{prop:scalas_correlated_Lipschitz}
    Let the $X^n$ be uncoupled correlated CTRWs as in \eqref{defi:CTRW} with $\alpha \in (0,2]$ and $\beta \in (0,1)$. Moreover assume that  \eqref{eq:technical_condition} holds, $(f_n)_{n\ge 1} \subseteq C^1$ with $\dM(f_n,f) \to 0$ and such that $|f_n'|^{*}_T \le C_T \, n^\gamma $ for each $n\ge 1$ and $T>0$, constants $C_T>0$ which only depend on $T$ as well as some $\gamma \in (0,\beta/\alpha)$. Then, the convergence in \eqref{eq:conv_scalas} holds with $X=(\sum_{j=0}^\infty c_j) Z_{D^{-1}}$.
\end{prop}

In fact, it suffices that the $f_n$ are Lipschitz on compacts $[0,T]$ with Lipschitz constant $C_T n^\gamma$. Also---provided AVCI is satisfied for the integrands and integrators (see \cite[Def.~3.2 \& Prop.~3.8]{andreasfabrice_theorypaper})---there is no need to restrict to deterministic functions as long as they almost surely exhibit this Lipschitz property, are adapted to \eqref{eq:filtration_CTRWs} and converge weakly to some càdlàg limit.

\section{Cointegrated processes in econometrics}\label{subsect:econometrics}

Exactly as Paulaskas \& Rachev \cite{paulaskas_rachev}, we shall focus on the general formulation of multivariate cointegrated time series from the classical work of Park \& Phillips \cite{park_phillips}. Given a $k$-dimensional time series $(Y_{t_i})_{i= 1}^n$ and an $m$-dimensional time series $(X_{t_i})_{i= 1}^n$, for some time interval $[0,t]$, we are interested in understanding if we have the structure
\[
 X^n_{t_i} = X^n_{t_{i-1}}  + v_{t_i} \quad \text{and} \quad Y^n_{t_i} = \mathbb{M}^n X^n_{t_i} + u_{t_i} \quad  \text{for}\quad {t_1,\ldots, t_n}\in[0,t],
\]
for some $k\times m$ matrix $\mathbb{M}^n$ and suitable innovations $(u_{t_i})_{i=1}^n$ and $(v_{t_i})_{i=1}^n$. As in \cite[Sect.~3]{park_phillips}, one can then show that the least square estimator $\hat{\mathbb{M}}^n$ of $\mathbb{M}^n$ is of the form
\begin{equation*} \hat{\mathbb{M}}^n = \mathbb{Y}^\intercal \mathbb{X} ( \mathbb{X}^\intercal  \mathbb{X})^{-1}, \quad \text{with} \quad \hat{\mathbb{M}}^n - \mathbb{M}^n= \mathbb{U}^\intercal \mathbb{X} ( \mathbb{X}^\intercal  \mathbb{X})^{-1},
\end{equation*}
where $\mathbb{X} $, $\mathbb{Y}$, and $ \mathbb{U}$ are appropriate matrices collecting the two time series and the first innovation sequence (dependence on $t$ and $n$ has been suppressed). To formulate hypothesis tests, the idea is to look for an invariance principle for $\hat{\mathbb{M}}^n - \mathbb{M}^n $ as the sample size $n$ tends to infinity, meaning that we pass to observations in continuous time in $[0,t]$. Naturally, this will rely on first assuming some invariance principle for the innovations $(u_{t_i})_{i=1}^n$ and $(v_{t_i})_{i=1}^n$, and it will involve suitable scaling factors depending on the assumption on the laws of these innovations.

While \cite{park_phillips} considers Gaussian invariance principles, the objective of \cite{paulaskas_rachev} (and the related works \cite{caner2,caner1}) is to allow for heavy-tailed innovations with infinite variance. Based on stable limit theorems, \cite{paulaskas_rachev} thus consider the re-scaled differences
\begin{equation}\label{eq:estimator_error} n \mathbb{T} ( \hat{\mathbb{M}}^n - \mathbb{M}^n ) \tilde{\mathbb{T}}^{-1} = 
	\bigl(\mathbb{T} \mathbb{U}^\intercal \mathbb{X} \tilde{\mathbb{T}} \bigr) \bigl(\tfrac{1}{n} \tilde{\mathbb{T}} \mathbb{X}^\intercal \mathbb{X} \tilde{\mathbb{T}}   \bigr)^{-1},
\end{equation}
where
\[
\mathbb{T} = \textrm{diag}(\tfrac{1}{n^{1/\alpha_1}},\ldots, \tfrac{1}{n^{1/\alpha_p}})\quad \text{and} \quad \tilde{\mathbb{T}} = \textrm{diag}(\tfrac{1}{n^{1/\tilde{\alpha}_1}},\ldots, \tfrac{1}{n^{1/\tilde{\alpha}_q}}),
\]
for suitable parameters $1 < \alpha_i,\tilde{\alpha}_j \leq 2$, $i=1,\ldots,p$, $j=1,\ldots,q$. Rewriting $\mathbb{T} \mathbb{U}^\intercal \mathbb{X} \tilde{\mathbb{T}}$ in \eqref{eq:estimator_error}, it turns out that the main challenge towards an invariance principle for \eqref{eq:estimator_error} is the joint weak convergence of the stochastic integrals $\int_0^\bullet Z^{n,i}_{s-} \diff \tilde{Z}^{n,j}_{s}$, for $i=1,\ldots p$ and $j=1,\ldots, q$, where
\begin{equation}\label{eq:random_walks}
	Z^{n,i}_s := \sum_{\ell=1}^{\lfloor ns/t \rfloor} \frac{1}{n^{1/\alpha_i}} u^{(i)}_{t_\ell}\quad \text{and} \quad \tilde{Z}^{n,j}_s := \sum_{\ell=1}^{\lfloor ns/t \rfloor} \frac{1}{n^{1/\tilde{\alpha}_j}} v^{(j)}_{t_\ell},\quad \text{for}\quad s\in [ 0,t].
\end{equation}

In \cite{paulaskas_rachev}, the paired innovations $(u_{t_\ell},v_{t_\ell})_{\ell=1}^n$ are assumed to form an i.i.d.~sequence of $(p+q)$-dimensional random variables in the normal domain of attraction of a stable law on $\mathbb{R}^{p+q}$ with parameters $(\alpha,\tilde{\alpha})=(\alpha_1,\ldots,\alpha_p,\tilde{\alpha}_1,\ldots, \tilde{\alpha}_q)$. In this setting, the processes \eqref{eq:random_walks} converge weakly on $(\mathbf{D}_{\mathbb{R}^{p+q}}[0,t],\mathrm{d}_{\text{J1}})$ to a pair of $(\alpha,\tilde{\alpha})$-stable L{\'e}vy processes $(Z,\tilde{Z})$ and the desired stochastic integral convergence becomes a direct consequence of the general framework in \cite{andreasfabrice_theorypaper}.

\begin{prop}\label{prop:econometrics_1}
    Let $Z^{n,i}, \tilde{Z}^{n,j}$ be as in \eqref{eq:random_walks} with the above assumptions on the innovations satisfied for some $\alpha_i,\tilde{\alpha}_j\in(0,2]$, $i=1,\ldots,p$, $j=1,\ldots,q$. Then there is weak convergence
    $$\Big( Z^n, \tilde{Z}^n, \int_0^\bullet \, Z^{n}_{s-} \; \otimes \diff \tilde{Z}^{n}_{s}\Big) \; \Rightarrow \; \Big( Z, \tilde{Z}, \int_0^\bullet \, Z_{s-} \; \otimes \diff \tilde{Z}_{s}\Big) \qquad \text{on } \; \; (\mathbf{D}_{\mathbb{R}^{p+q+pq}}[0,t],\, \dJ)$$
    where the outer product $\int_0^\bullet  Z^{n}_{s-}  \otimes \diff \tilde{Z}^{n}_{s}$ denotes the matrix $(\int_0^\bullet  Z^{n,i}_{s-}\diff \tilde{Z}^{n,j}_{s})_{i,j}$.
\end{prop}
\begin{proof}
    Due to the assumptions, the $(Z^n,\tilde{Z}^n)$ are CTRWs converging together to $(Z,\tilde{Z})$ on $(\mathbf{D}_{\mathbb{R}^{p+q}}[0,t],\dJ)$ by the results mentioned in Section \ref{subsect:CTRW}. In particular, they satisfy the AVCI condition from \cite[p.8]{andreasfabrice_theorypaper} in view of \cite[Prop.~3.8(1)]{andreasfabrice_theorypaper}. Since the sequences $(\tilde{Z}^{n,j})_{n\geq 1}$ admit good decompositions by \cite[Thm.~3.25]{andreasfabrice_theorypaper}, the result follows from \cite[Thm.~3.6 \& Rem.~3.11]{andreasfabrice_theorypaper}.
\end{proof}

We note that the arguments in \cite{paulaskas_rachev} only apply to $\alpha_i,\tilde{\alpha}_j \in (1,2]$. The main technical work is concerned with verifying the UCV condition of \cite{kurtzprotter} in order to apply \cite[Thm.~2.7]{kurtzprotter} (see Remark \ref{rem_UCV} below). Whilst not explicit, a close inspection of the proof of \cite[Prop.~3]{paulaskas_rachev} reveals that these arguments boil down to a bound on $\mathbb{E}^n[\operatorname{sup}_{s\in[0,t]}|\Delta\tilde{Z}^{n}_s|]$ uniformly in $n\geq 1$ (something that is specific to the range $\alpha_i,\tilde{\alpha}_j \in (1,2]$\,). As the 
$\tilde{Z}^{n}$ can be seen to be martingales in this range, they immediately satisfy the martingale condition for good decompositions in \cite[Def.~3.3]{andreasfabrice_theorypaper}, and hence we have convergence of the stochastic integrals without any additional work.

\begin{remark}\label{rem_UCV}The arguments in \cite{paulaskas_rachev} rely on \cite[Thm.~2.7]{kurtzprotter} by finding decompositions $\tilde{Z}^n=M^n+A^n+J^n$ for which they can (i) verify the UCV condition for $M^n$ and $A^n$, and (ii) show that the $Z^n$, $\tilde{Z}^n$, and $J^n$ converge together in J1, where the $J^n$ are constant except for finitely many jumps in $[0,t]$, for each $t>0$, and have a number of jumps in $[0,t]$ that is tight, for each $t>0$. It is worth highlighting here that the convergence requirement on the $J^n$ only appears because of the particular method of proof in \cite{kurtzprotter}. Since $\tilde{Z}^n$ is J1 tight with $M^n$ and $A^n$ satisfying the UCV property, one automatically gets tightness of the jump sizes of $J^n$, so (i) and (ii) give that the $J^n$ are of tight total variation, and hence they can immediately be put into the finite variation part of a good decomposition, regardless of any convergence issues.
\end{remark}

Results similar to \eqref{prop:econometrics_1} are central to a wide-ranging literature on cointegration tests, see e.g.~the recent book \cite{Wang_book}. Directly related to the above, the more general aymptotic analysis of heavy-tailed near-integrated time series with autoregression in \cite{chan_zhang} also relies on the arguments of \cite{paulaskas_rachev} through the proof of the key \cite[Lem.~A.4]{chan_zhang}. It should be noted that marginal convergence is often sufficient, but functional convergence can be of independent interest and we refer to ~\cite{davis, davis2} for cases where functional convergence plays an essential role.

Finally, we wish to highlight how the framework of \cite{andreasfabrice_theorypaper} can also shed light on cases where the innovations are given by linear processes (violating the assumptions of \cite{paulaskas_rachev}) and where a nonlinear function is involved. Consider, for example, stochastic integrals $\int_0^\bullet f(Z^{n}_{s-}) \diff \tilde{Z}^{n}_{s}$, for a continuous function $f$, where $Z^{n}$ and $\tilde{Z}^{n}$ are as in \eqref{eq:random_walks}. If the innovations $u$ and $v$ have the form of $\zeta$ from \eqref{defi:CTRW}, then $Z^{n}$ and $\tilde{Z}^{n}$ become strictly M1 convergent moving averages. It is straightforward to show that $(f(Z^n),\tilde{Z}^n)$ satisfy the weaker convergence assumptions outlined in \cite[Prop.~3.22]{andreasfabrice_theorypaper}. Depending on the application, it can be natural for the innovations $(u_{t_\ell})_{\ell \geq 1}$ and $(v_{t_\ell})_{\ell \geq 1}$ to model two independent sources of randomness: in such a case, we obtain the following functional convergence result. For simplicity of notation, we take $i=j=1$.
\begin{prop}
    Let $Z^{n}, \tilde{Z}^{n}$ be defined as in \eqref{eq:random_walks} and let the innovations be of the form
    $$u_{t_\ell}\, = \, \sum_{k=0}^\infty \, c_{k} \, \theta_{\ell-k} \qquad \text{ and } \qquad v_{t_\ell}\, = \, \sum_{k=0}^\infty \, \tilde c_{k} \, \tilde\theta_{\ell-k}$$ 
    where the $\theta_j$, $\tilde \theta_k$ are i.i.d.~random variables in the normal domain of attraction of an $\alpha$-- and $\tilde{\alpha}$--stable law respectively, the sequences $(c_{k})_{k\ge 1}, (\tilde c_{k})_{k\ge 1}$ satisfying the conditions given after \eqref{defi:CTRW} as well as \eqref{eq:technical_condition}, and the sequences $(\theta_k)_{k\ge 1}$ and $(\tilde \theta_k)_{k\ge 1}$ are independent. Let $f:\mathbb{R}\rightarrow\mathbb{R}$ be continuous. Assume $\E[\theta_0]=0$ if $\tilde{\alpha} \in(1,2]$ or that the law of $\theta_0$ is symmetric around zero if $\tilde{\alpha}=1$. Then, $\tilde{Z}$ is a semimartingale with respect to the filtration generated by the pair $(Z,\tilde Z)$, and there is weak convergence
    $$\Big( \tilde{Z}^n, \int_0^\bullet \, f(Z^{n}_{s-}) \; \diff \tilde{Z}^{n}_{s}\Big) \; \Rightarrow \; \Big( \tilde c \tilde{Z},\,  \tilde c \int_0^\bullet \, f(c  Z_{s-}) \;  \diff \tilde{Z}_{s}\Big) \quad \text{on} \quad (\mathbf{D}_{\mathbb{R}^{2}}[0,t],\, \dM)$$
    with $c:=\sum_{k=0}^\infty c_{k}$ and $\tilde c:=\sum_{k=0}^\infty \tilde c_{k}$.
\end{prop}
\begin{proof}
    The weak convergence of the $(Z^n, \tilde{Z}^n)$ on $(\mathbf{D}_{\mathbb{R}}[0,t],\, \dM)\times (\mathbf{D}_{\mathbb{R}}[0,t],\, \dM)$ follows from \eqref{eq:CTRW_M1_conv}, and it is readily verified that the $(f(Z^n),\tilde{Z}^n)$ satisfy the conditions F1-F3 of \cite[Prop.~3.22]{andreasfabrice_theorypaper}. Therefore, the claim follows directly from \cite[Rem.~4.16]{andreasfabrice}.
\end{proof}

\section{Stability of the financial gain in mathematical finance}\label{sect:asset_prices}

Famously, already in the 60's, Mandelbrot \cite{mandlebrot} and Fama \cite{fama} advocated for $\alpha$-stable distributions in finance. In particular, \cite{mandlebrot} illustrated how $\alpha =1.7$, rather than $\alpha=2$, gave a much better correspondence with observed time series for cotton prices. Inspired by such observations and driven by a phenomenological approach to market microstructure with random inter-trade durations, there has been a growing interest in CTRW models of asset prices, first formalized by Scalas, Gorenflo and Mainardi in \cite{scalas_finance} (building on the classical works \cite{clark,parkinson,press}).

Following on from \cite{scalas_finance}, we consider a sequence of asset prices $S^n_t$ that are determined, in a `tick-by-tick' fashion, by random price moves occurring only at unevenly-spaced random times $(t_i)_{i\geq 0}$. For each $n\geq 1$, we work with waiting times $t_i-t_{i-1}=n^{-1}J_i$ and log-price innovations $\ln(S^n_{t_i})-\ln(S^n_{t_{i-1}})=n^{-\beta/ \alpha}\zeta_i+n^{-1}J_i\,r$, where $r$ is the risk free rate which is reflected as an upward trend in the price innovations (investors should expect to be compensated for the risk free return that they could have earned by holding the risk free asset). By summing up these tick-by-tick price moves, we obtain the exponential CTRWs
\begin{equation}\label{eq:CTRW_asset_price}
S^n_t=e^{R^n_t},\quad \text{where}\quad  R^n_t := \sum_{i=1}^{N(nt)} \bigl(n^{-\beta/ \alpha}\zeta_i+ n^{-1}J_i\,r\bigr),
\end{equation}
with $N(t)$ as in \eqref{eq:CTRW_N}. Here, the components $\zeta_i$ of the log-price innovations are assumed to be i.i.d.~in the normal domain of attraction of an $\alpha$-stable law, for $\alpha \in [1,2]$, and the waiting times $J_i$ are assumed to be i.i.d.~in the normal domain of attraction of a $\beta$-stable law with $\beta \in (0,1)$, as in the standard literature on CTRW models for tick-by-tick prices. Note that, for all $t\ge 0$,
\begin{equation}\label{eq:sums_risk-free} rt \, - \, rJ_{N(nt)}/n \; < \; \sum_{i=1}^{N(nt)}n^{-1}J_i\,r \; \le \; rt,
\end{equation}
and $J_{N(nt)}/n \Rightarrow 0$, so we readily obtain that the increasing sums in \eqref{eq:sums_risk-free} are weakly M1 convergent to the continuous limit $t\mapsto rt$ (hence weakly uniformly convergent). Writing $\tilde{R}^n_t$ for the sum of the $n^{-\beta/ \alpha}\zeta_i$ terms in \eqref{eq:CTRW_asset_price}, we know from Section \ref{subsect:CTRW} that the CTRWs $\tilde{R}^n$ converge weakly in J1 to $\tilde{R}:=Z_{D^{-1}}$ or $\tilde{R}:=((Z^-)_{D^{-1}})^+$, for an $\alpha$-stable L\'evy process $Z$ and $\beta$-stable subordinator $D$, depending on whether they are coupled or uncoupled. As one limit is continuous, we conclude that $R^n \Rightarrow \tilde{R}_\bullet+r\bullet$ in J1 and then also $S^n \Rightarrow S:=e^{\tilde{R}_\bullet+r\bullet}$ in J1.

For the limiting price $S$, the $\beta$-stable subordinator $D$ with $\beta\in(0,1)$ allows one to capture burst of activity in between periods of relative dormancy, and it yields price paths that are locally constant at almost all times in line with the intended tick-by-tick interpretation of price moves. For option pricing, one needs $\tilde{R}_t$ to have an exponential moment so that the discounted asset price $e^{-rt}X_t=e^{\tilde{R}_t}$ can be a martingale. This can be realised by `tempering' the probability of large innovations, leading to the notion of tempered stable L\'evy processes. To this end, \cite{Chakrabarty_Meerschaert} has shown that one can take an i.i.d.~sequence of innovations in the domain of attraction of a stable law, as in \eqref{eq:CTRW_asset_price}, and truncate these jumps in a certain random way (mimicking the tempering function for the desired limit) such that one obtains a continuous-time random walk for which a suitable (deterministic) scaling yields convergence to the desired tempered stable L\'evy process. Combining the triangular scheme of \cite{Chakrabarty_Meerschaert} with the results in \cite{henry_straka,jurlewicz}, this also works with random waiting times resulting in a subordinated limit (see in particular \cite[Rem.~3.5]{jurlewicz}). 

Jacquier \& Torricelli \cite{jack} have developed a remarkably tractable option pricing theory for two such limiting  models, motivated by the CTRW formalism (with appropriate tempering) and capitalising on analytical results of \cite{Becker-Kern, jurlewicz,leonenko,meerschaert_straka,meerschaert}. The first is the \emph{subdiffusive L\'evy} model, corresponding to $\tilde{R}=Z_{D^{-1}}$, where $Z$ and $D$ are independent (many earlier works are special cases of this). The second is the \emph{dependent returns and trade duration} model, corresponding to $\tilde{R}=((Z^-)_{D^{-1}})^+$, where $Z$ and $D$ are dependent but enjoy a particular structure $Z=\tilde{Z}_D$ for a L\'evy process $\tilde{Z}$ independent of $D$. The point of the latter is that one then has $\tilde{R}=\tilde{Z}_G$ up to indistinguishability with $G:=g^+$, for the last-passage time $g(t):=\text{sup} \{ s < t : s\in \{D_r:r\geq 0\} \}$, where the new (right-continuous, increasing) time-change $G$ is now independent of the parent $\tilde{Z}$ and one can exploit powerful analytical results for the law of $g$. In \cite{jack}, these models are shown to produce volatility surfaces that capture important phenomena observed in the market much better than classical L\'evy or stochastic volatility models. As in \cite[Ex.~5.4]{jurlewicz}, one can obtain $\tilde{R}=\tilde{Z}_G$ as the limit directly, by using a suitable triangular array scheme.

\subsection{On J1 stability of the financial gains}

Similarly to the analysis of Duffie \& Protter \cite{duffie_protter} for the Black--Scholes model, it is important to understand the stability properties of financial gain processes $\int_0^\bullet H^n_{t-}\diff S^n_t$, for trading strategies $H^n_{-}$, as the asset prices $S^n$ converge to their scaling limit. With $S^n$ as in \eqref{eq:CTRW_asset_price}, so that $S^n\Rightarrow S =e^{\tilde{R}_\bullet+r\bullet}$, we obtain the following stability result for the financial gains.
\begin{prop}\label{prop:fin_stability_J1}
    Let $H^n$ be càdlàg processes adapted to the filtrations \eqref{eq:filtration_CTRWs} such that $(H^n,S^n) \Rightarrow (H, S)$ on $(\D_{\R^{2d}}[0,\infty) , \dJ)$. Then, $S=e^{\tilde{R}_\bullet+r\bullet}$ is a semimartingale with respect to the filtration generated by the pair $(H,S)$ and it holds that
    $$\Big(H^n, \, S^n, \, \int_0^\bullet \, H^n_{s-} \; \diff S^{n}_{s}\Big) \; \Rightarrow \; \Big( H,\,  S,\,  \int_0^\bullet \, H_{s-} \;  \diff S_s \Big) \qquad \text{on } \; \; (\mathbf{D}_{\mathbb{R}^{3d}}[0,\infty),\, \dJ).$$
\end{prop}
\begin{proof}
 First note that the $R^n$ from \eqref{eq:CTRW_asset_price} have good decompositions due to \cite[Thm.~3.25]{andreasfabrice_theorypaper} and \eqref{eq:sums_risk-free}, so \cite[Rem.~3.20]{andreasfabrice_theorypaper} gives that also the $S^n$ have good decompositions. As a consequence of the convergence together of the pairs $(H^n,S^n)$, AVCI is satisfied by 
\cite[Prop.~3.8(1)]{andreasfabrice_theorypaper}, and so \cite[Thm.~3.6]{andreasfabrice_theorypaper} yields the conclusion.
\end{proof}

\begin{remark}
    Following \cite[Sect.~5]{duffie_protter}, the convergence together of $(H^n, S^n)$ on $(\D_{\R^{2d}}[0,\infty)\, ,\, \dJ )$ can, for example, be seen to be valid in the following simple cases, based on continuous functions $h_n(t,x)$ converging uniformly to some (continuous) $h(t,x)$. We can have $H^n_t=h_n(t,S^n_{t})$ and $H_t=h(t,S_t)$. Or, we can have that the $H^n$ are discretisations of $h_n(t,S^n_{t})$ of the form
    $$ H^n_t \; := \; \sum_{k=0}^\infty \, h_n(\tau^n_k, S^n_{\tau^n_k}) \, \ind_{[\tau^n_k,\tau^n_{k+1})}(t)$$
    where $0\equiv \tau^n_0< \tau^n_1< \tau^n_2...$ are $S^n$--stopping times such that, for all $n\ge 1$, $\tau^n_k \to \infty$ as $k\to \infty$ almost surely, and the mesh size $\operatorname{sup}_{k\ge 1} |\tau^n_k-\tau^n_{k-1}|\to 0$ as $n\to \infty$.
\end{remark}

We note that, in Proposition \ref{prop:fin_stability_J1}, one could also consider weakly M1 convergent trading strategies or, more generally, trading strategies that converge in the sense of Proposition \cite[Prop.~3.22]{andreasfabrice_theorypaper}. However, then one must check that AVCI or its alternatives from \cite{andreasfabrice_theorypaper} are satisfied, as this is no longer automatic. Since it is natural to allow the limiting (left continuous) trading strategies to jump at the same time as the limiting (right continuous) price processes, the simple sufficient criterion in \cite[2 of Prop.~3.8]{andreasfabrice_theorypaper} is unlikely to be helpful here. However, it can then be useful to seek to verify the conditions of \cite[Thm.~4.8]{andreasfabrice_theorypaper} for the interplay of the $H^n$ and $S^n$ (as an alternative to AVCI) in order to deduce the J1 stability of the financial gains. 

In order to make the above precise, for any given process $Y$, define $\operatorname{Disc}_{\Pro}(Y):=\{s>0: \Pro(\Delta Y_s\neq 0)>0\}$ and the number of its $\delta$-increments, $$N_{\delta}^T(Y):= \operatorname{sup}\{m\ge 1\, : \, \exists (t_i)_{i=1}^{2m}\subseteq [0,T], \, t_{i}< t_{i+1}, \text{ such that } |Y_{t_{2i+1}}-Y_{t_{2i}}|>\delta\}.$$ Further, we shall write $(H^n, X^n) \Rightarrow_{f.d.d.\times\text{J1}} (H,X)$ on $D\subseteq [0,\infty)$ if
$$ ( H^n_{t_1},...,H^n_{t_\ell}, X^n) \; \Rightarrow \; (H_{t_1},...,H_{t_\ell},X) \qquad \text{ on } \quad (\R^{\ell d}, |\cdot|\, )\times (\mathbf{D}_{\mathbb{R}^{d}}[0,\infty),\, \dJ)$$
for all $\ell\ge 1$ and $t_1,...,t_\ell \in D$.

\begin{prop}
    Let the $\tilde{R}^n$ be uncorrelated uncoupled CTRWs (i.e.,~the $\zeta_k$ and $J_i$ are independent), and consider càdlàg processes $H^n$ adapted to the filtrations \eqref{eq:filtration_CTRWs}. Suppose there is a dense subset $D\subseteq [0,\infty) \setminus \operatorname{Disc}_{\Pro}(H,P)$ such that $(H^n, S^n) \Rightarrow_{f.d.d.\times\text{J1}} (H,S)$ on $D$ and the sequences $(|H^n|^*_T)_{n\ge 1}, (N^T_\delta(H^n))_{n\ge 1}$ are tight for each $T>0$ and $\delta>0$. Then, $S=e^{\tilde{R}_+r\bullet}$ with $\tilde{R}=Z_{D^{-1}}$ is a semimartingale for the filtration generated by the pair $(H,R)$ and it holds that
    $$\Big( S^n, \int_0^\bullet \, H^n_{s-} \; \diff S^{n}_{s}\Big) \; \Rightarrow \; \Big( S,\,  \int_0^\bullet \, H_{s-} \;  \diff S_s \Big) \qquad \text{on } \; \; (\mathbf{D}_{\mathbb{R}^{2d}}[0,\infty),\, \dJ).$$
\end{prop}
\begin{proof}
    As in the proof of Proposition \ref{prop:fin_stability_J1}, the $S^n$ have good decompositions. Towards an application of \cite[Thm.~4.8]{andreasfabrice_theorypaper}, according to \cite[Rem.~4.9]{andreasfabrice_theorypaper} it suffices to check its assumptions (a)\&(b) for $H^n$ and $R^n$. Those are indeed satisfied, where part (a) can be readily shown by means of \cite[Lem.~4.12]{andreasfabrice_theorypaper} due to the adaptedness of the $H^n$ to the natural filtration of the $R^n$. 
\end{proof}

For specific trading strategies, such as the type displayed in \cite[Eq. (4.2)]{andreasfabrice_theorypaper} with $J=0$, it turns out that one can also readily verify AVCI, if for example the $g_{n,i}$ are equicontinuous. Indeed, using this property and the J1 tightness criteria given in \cite[Thm.~12.4]{billingsley}, we see that AVCI follows directly from the J1 tightness of the integrators. Furthermore, this then also holds for uncorrelated coupled CTRWs due to \eqref{eq:coupled_CTRW_J1_conv}.

\subsection{On M1 stability of the financial gains}

Whilst not considered in \cite{jack}, it is worth briefly examining what happens to the stability of the financial gains if the price processes exhibit a correlation structure with respect to the innovations and therefore are weakly M1 convergent (but not J1 convergent). Indeed, the picture changes rather drastically in this case. Albeit in a different setting from the CTRW framework considered here, we note that A{\"i}t-Sahalia \& Jacod \cite{ait-sahalia_jacod} define limiting semimartingale asset prices to be compatible with tick-by-tick models precisely if there is weak M1 convergence.

Suppose the log-prices $R^n$ in \eqref{eq:CTRW_asset_price} have innovations $\zeta_i$ of the form \eqref{defi:CTRW} with $c_j>0$ for at least two indices so that we have the strict M1 convergence \eqref{eq:CTRW_M1_conv}. Up to the appropriate rescaling by a constant, this results in the same form of the limiting asset price $S$ and hence gives rise to the same option pricing theory as in \cite{jack}. However, the stability of the financial gains is now lost in general. Indeed, for $1<\alpha<2$, \cite[Prop.~4.6]{andreasfabrice_theorypaper} shows that there exist admissible trading strategies, almost surely uniformly vanishing in the limit, for which there is nevertheless a blow-up of the associated financial gain. Intuitively, this lack of stability is coherent with the autoregressive correlation structure of the price innovations, as this in effect opens up the floodgates to anticipating---and hence exploiting---the nature of future price moves (noting also that the total variation of the price processes is not tight).

In the spirit of the Bachelier model, `arithmetic' L\'evy models (as opposed to `geometric'/exponential Lévy models) are also considered in the financial literature. For this final part of the section, we switch to consider `arithmetic' CTRW price processes $X^n$, where the $X^n$ are correlated CTRWs as in \eqref{defi:CTRW} with strict M1 convergence to a (sub-diffusive) subordinated $\alpha$-stable L\'evy process. As above, stability of the financial gains is lost (owing to a lack of good decompositions for the $X^n$), but the particular structure of having the $X^n$ as integrators nevertheless allows us to derive simple and intuitive sufficient criteria for recovering stability. For example, we can restrict the class of admissible trading strategies $H^n$ in such a way that, at any time, they are independent of the lagged price innovations that are part of determining the future price innovations: that is, if the lagged dependence of $X^n$ stretches over $\mathcal J+1$ price moves, then the portfolio at time $t$ is based solely on decisions made at least $\mathcal J+1$ prices moves ago. As $n\rightarrow \infty$, this window converges to zero, so, asymptotically, it is a mild restriction on how fast an investor can act relative to how the market moves.

To be specific, a class of trading strategies $H^n$ that are admissible in the sense of the above discussion are processes of the form\\[-2ex]
\[\label{eq:Hn_fin_application} H^n_t \; := \; \sum_{i=1}^{\infty} \, g_{n}(t^n_i,X^n_{ \Xi^n(t^n_i)}) \, \ind_{[t^n_i, \, t^n_{i+1})}(t)
\]
 with $\Xi^n(t^n_i):= \operatorname{max}\{\sigma^n_{k-\mathcal J-1} \, : \, k\ge 1 \text{ with } k-\mathcal J-1\ge 1, \, \sigma^n_k\le t^n_i\}$, where $\sigma^n_k=L_k/n$ is the $k$-th jump time of $X^n$, and where $t^n_i\uparrow \infty$ as $i \to \infty$,  $\operatorname{sup}_{i}|t^n_{i+1}-t^n_i| \to 0$ as $n\to \infty$ (we set $X^n_{\operatorname{max}\{\emptyset\}}:=0$). If the $g_{n}$ are sufficiently well-behaved, e.g. continuous and uniformly converging to some continuous, non-decreasing map $g$, then it is not hard to deduce the weak convergence of the $H^n$ on the M1 space from that same convergence of the $X^n$. In addition, we take the i.i.d.~innovations of the CTRWs $X^n$ to be centered for $1<\alpha<2$ and their common law to be symmetric if $\alpha=1$. The particular form of the $H^n$ directly implies \cite[(4.8)]{andreasfabrice} and since \eqref{eq:technical_condition} is trivially satisfied, on behalf of \cite[Thm.~4.15]{andreasfabrice} we obtain the convergence $(X^n, \int_0^\bullet H^n_{s-} \diff X^n_s) \Rightarrow (X, \int_0^\bullet H_{s-} \diff X_s)$ in M1 with $X=(\sum_{j=0}^\mathcal{J} c_j)Z_{D^{-1}}$ and $H_t=g(t,X_{t-})$. 

Alternatively, there is another intuitive approach to recovering the stability of the financial gains in the framework of finitely lagged CTRWs (i.e., $c_j=0$ for all $j> K$, for some $K$). Indeed, one can require the trading strategies $H^n$ to exhibit sufficient inertia in a continuous fashion, thus making them `too gradual' in their reaction to critical new price moves. More details on this particular approach can be found in \cite[Thm. 4.14]{andreasfabrice}.

\section{Reserves with risky investments in insurance mathematics}
\label{subsect:insurance}

Consider an insurance company whose total cash flow of premiums net of claims is described by a stochastic process $P_t$. Assume further that this company is investing the entirety of its current reserves, denoted by $U_t$, in a risky asset (or index fund) with price $S_t$ at time $t$. Then, the reserves will evolve according to the SDE
\begin{equation}\label{eq:total_reserves_SDE}
	\diff U_t = U_{t-} \frac{\diff S_t}{S_{t-}} + \diff P_t.
\end{equation}
Following Paulsen \cite[Sect.~2]{Paulsen_AAP} and Kalashnikov \& Norberg \cite[Sect.~2.2]{Kalashnikov_Norberg}, consider two independent L\'evy processes $R$ and $P$, where $R=\log (S)$, and note that the solution to \eqref{eq:total_reserves_SDE} is then given by
\begin{equation}\label{eq:total_reserves_solutoin}
	U_t = e^{R_t} \Bigl(U_0+ \int_0^t e^{-R_{s-}} \diff P_s\Bigr),\quad t\geq 0.
\end{equation}
As in the previous section, motivated by a CTRW model, one could also consider modelling the log-return $R$ by a time-changed L\'evy process. Similarly, we refer to Constantinescu, Ramirez \& Zhu \cite{constantinescu} for examples of CTRWs with non-exponential waiting times in the modelling of claims, which could lead to $P$ being modelled by a time-changed L\`evy process (possibly with an additional non-time-changed component for the premiums, which we leave out here for simplicity of notation). Maintaining the independence of $P$ and $R$ (due to independence of the financial and insurance risk), one still has $[P,R]=0$, so the representation \eqref{eq:total_reserves_solutoin} continues to hold.

Now suppose there is an underlying CTRW model, where the log-returns $R^{n}$ and the net cash flows $P^{n}$ are given by moving averages or uncoupled CTRWs \eqref{defi:CTRW}, either uncorrelated or correlated, e.g.~to capture large losses followed by further large losses. As before, we take $R^n$ and $P^n$ to be independent. Allowing for heavy-tailed waiting times and heavy-tailed innovations, the usual assumptions (see Section \ref{subsect:CTRW}) give that $R^n$ and $P^n$ will be weakly convergent, in the M1 or J1 topology, to a pair of independent, possibly subordinated stable Lévy processes. We denote these limits by $P$ and $R$, in line with the above. Furthermore, for each $n\geq 0$, we write $S^n:=\operatorname{exp}(R^n)$ for the price process and note that the corresponding reserves are given by
\[
U^n_t= e^{R^n_t}\Bigl(U_0 + \int_0^\bullet e^{-R^n_{s-}} \diff P^n_s\Bigr),\quad t\geq0.
\]

The following result makes precise the realisation of \eqref{eq:total_reserves_SDE}-\eqref{eq:total_reserves_solutoin} as a limit of the CTRW models just described. We note that a similar result for limiting processes $R$ and $P$ given by Brownian motion with drift was established in \cite{Paulsen_Gjessing_Advances}.

\begin{prop}
    Let $R_n$, $P_n$ be two independent correlated uncoupled CTRWs of the form \eqref{defi:CTRW} so that $(R^n,P^n) \Rightarrow (R,P)$ on $(\D_{\R}[0,\infty), \tilde\rho) \times (\D_{\R}[0,\infty), \rho)$ with $\tilde\rho,\rho \in \{\dJ,\dM\}$ (depending on whether the CTRWs are correlated or not). If the $P^n$ are strictly M1 convergent correlated CTRWs, assume that \eqref{eq:technical_condition} holds, and assume that the innovations are centered if $1<\alpha<2$ or symmetric around zero if $\alpha=1$. Then, we have
    \begin{align} 
    (S^n, P^n, U^n) \; \Rightarrow \; (S, P, U) \quad \text{ on } \quad (\D_{\R}[0,\infty), \dM) \times (\D_{\R^{2}}[0,\infty), \dM) \label{eq:insurance_convergence_goal}
    \end{align}
    and the convergence holds on $(\D_{\R^{3}}[0,\infty), \dJ)$ if $\rho=\tilde\rho=\dJ$.
\end{prop}
\begin{proof}
    Due to $s\mapsto \exp(-s)$ being uniformly continuous on compacts and $x_n \to x$ on the J1 space, it directly follows from the definition of standard metric on the J1 space that also $\exp(-x_n) \to \exp(-x)$ in J1. Since, in addition, $s\mapsto \exp(-s)$ is monotone, convergence also propagates to M1. To see this, note that it suffices to show that the sequence $\{ \exp(-x_n)\}$ is relatively compact in M1 as pointwise convergence at times in a co-countable set holds as a consequence of the continuity the exponential function and the M1 convergence $x_n \to x$. Making use of \cite[Thm.~A.8]{andreasfabrice_theorypaper}, (a) is immediate and for (b) it clearly holds for any $0\le t_1<t_2<t_2\le T$, $t_3-t_1\le \theta$, 
    \begin{align*}
        &\norm{\, \exp\{-x_n(t_2)\}-[\exp\{-x_n(t_1)\},\exp\{-x_n(t_3)\}]\, } \\[1ex]
        & \qquad \quad \le \; 2 \operatorname{sup}\{|\exp(-z_1)-\exp(-z_2)| : |z_1|,|z_2|\le|x_n|^*_T, \, |z_1-z_2|\le \tilde{w}(x_n,\theta/2) \} \, =: \, \Lambda(x_n,\theta) 
    \end{align*}
    where we stress that $\Lambda$ does only depend on $\theta$ and $x_n$ and thus $\tilde{w}(\exp(-x_n), \theta/2) \le \Lambda(x_n,\theta)$. Given that the sequence $\{x_n\}$ is relatively compact on the M1 space, $\operatorname{sup}_{n\ge 1}|x_n|^*_T< \infty$ and $\lim_{\theta \downarrow 0}\operatorname{sup}_{n\ge 1} \tilde{w}(x_n)=0$. The uniform continuity of the exponential function then implies $\lim_{\theta \downarrow 0} \operatorname{sup}_{n\ge 1} \Lambda(x_n,\theta)=0$. Thus, we deduce the M1 relative compactness of $\{\exp(-x_n)\}$ and hence the continuity of $x\mapsto (s\mapsto \exp(-s))$ on the M1 space, implying---on behalf of the continuous mapping theorem---the convergence $(S_n, P^n) \Rightarrow (S,P)$ on $(\D_{\R}[0,\infty), \tilde\rho) \times (\D_{\R}[0,\infty), \rho)$.

    In addition, it turns out that the (subordinated) independent Lévy processes $R$ and $P$ almost surely do not jump at the same time. Indeed, each of the two processes is of the form $cZ_{D^{-1}}$ with $Z$ and $D$ independent Lévy processes (where independence is a result of the approximating CTRWs being uncoupled) as well as $D^{-1}$ the generalised inverse defined as in \eqref{eq:generalised_inverse}. If we can show that $R$, $P$ do almost surely not jump at fixed times, then it is well-known that, by independence, $R$ and $P$ almost surely have no common jumps. Towards this aim, fix $t>0$ and let $T>0$ as well as $\{\tau_m\}$ be a sequence of stopping times such that $\{0<s\le T: |\Delta Z_s|>0\} \subseteq \bigcup_{m=1}^\infty [\tau_m]$. The Lévy process $Z$ not having any fixed time discontinuities, $\Pro(\tau_m=s) \le \Pro(|\Delta Z_s|>0)=0$ and therefore each of the distribution functions $s\mapsto F_{\tau_m}(s):=\Pro(\tau_m\le s)$ is continuous. Hence, by independence of $\tau_m$ and $D^{-1}(t)$ and since $\Delta F_{\tau_m}(s)=0$ for all $s>0$,\\[-4ex]
    \begin{align}
        &\Pro(|\Delta Z_{D^{-1}(t)}|>0) \; \le \; \Pro(D^{-1}(t)> T) \; + \; \sum_{m=1}^\infty \Pro(\tau_m=D^{-1}(t)) \notag \\
        &\quad= \; \Pro(D^{-1}(t)> T) \; + \; \sum_{m=1}^\infty \, \int_0^\infty \int_0^\infty \ind_{\{s=r\}}(s,r) \, F_{\tau_m}(\diff s)\, F_{D^{-1}(t)}(\diff r) \notag \\
        &\quad= \; \Pro(D^{-1}(t)> T) \; + \; \sum_{m=1}^\infty \, \int_0^\infty \Delta F_{\tau_m}(r)\, F_{D^{-1}(t)}(\diff r) \; = \; \Pro(D^{-1}(t)> T),\label{eq:subordinated_levy_process_does_not_jump_at_deterministic_times}
    \end{align}
    where we have denoted $F_{D^{-1}(t)}$ the distribution function of the random variable $D^{-1}(t)$. The right side of \eqref{eq:subordinated_levy_process_does_not_jump_at_deterministic_times} tending to zero as $T\to \infty$, we deduce that with full probability, $Z_{D^{-1}}$ does not jump at any given deterministic time. Thus, $R$ and $P$ almost surely do not have any common discontinuities and consequently the pairs $(e^{-R^n}, P^n)$ satisfy AVCI due to \cite[Prop.~3.8(2)]{andreasfabrice_theorypaper}.

    If the $P^n$ are uncorrelated CTRWs and thus admit good decompositions by \cite[Thm.~3.25]{andreasfabrice_theorypaper}, an application of \cite[Thm.~3.6]{andreasfabrice_theorypaper} gives $(P^n,\int_0^\bullet e^{-R^n_{s-}} \diff P^n_s) \Rightarrow (P, \int_0^\bullet e^{-R_{s-}} \diff P_s)$ on $(\D_{\R^2}[0,\infty), \dJ)$ and thus $(S^n,P^n,\int_0^\bullet e^{-R^n_{s-}} \diff P^n_s) \Rightarrow (S,P, \int_0^\bullet e^{-R_{s-}} \diff P_s)$ on $(\D_{\R}[0,\infty), \tilde\rho)\times (\D_{\R^2}[0,\infty), \dJ)$. As multiplication is a continuous operation in J1 at all limit points \emph{without} common discontinuities (see \cite[Thm.~13.3.2]{whitt}), an application of the generalised continuous mapping theorem yields the desired convergence \eqref{eq:insurance_convergence_goal} as the $P,S$ almost surely do not have common discontinuities.

    Instead, if we consider correlated CTRWs $P^n$, one could in principle proceed in complete analogy, but now the $P^n$ will generally not admit good decompositions if $1< \alpha < 2$ (recall \cite[Prop.~4.6]{andreasfabrice_theorypaper}). Nevertheless, due to the independence of the processes $R^n$ and $P^n$, it is indeed possible to obtain \eqref{eq:insurance_convergence_goal}, via a suitable extension of the results considered in \cite{andreasfabrice_theorypaper}, namely \cite[Thm.~4.15]{andreasfabrice}, under the given assumption that the innovations of the $P^n$ are centered if $1<\alpha<2$ and their laws being symmetric if $\alpha=1$. 
\end{proof}

\medskip

\noindent\textbf{Acknowledgments.} We thank Lorenzo Torricelli for insightful discussions concerning the contents of Section \ref{sect:asset_prices}. The research of FW was funded by the EPSRC grant EP/S023925/1.


\printbibliography

\end{document}